\documentclass[final,hidelinks,onefignum,onetabnum]{siamart220329}

\usepackage{lipsum}
\usepackage{amsfonts}
\usepackage{graphicx}
\usepackage{epstopdf}
\usepackage{algorithmic}
\ifpdf
  \DeclareGraphicsExtensions{.eps,.pdf,.png,.jpg}
\else
  \DeclareGraphicsExtensions{.eps}
\fi

\usepackage{color}
\usepackage{mathtools}

\newtheorem{prop}[theorem]{{Proposition}}

\newtheorem{example}[theorem]{Example}

\newtheorem{remark}[theorem]{Remark}

\numberwithin{equation}{section}

\newsiamremark{hypothesis}{Hypothesis}
\crefname{hypothesis}{Hypothesis}{Hypotheses}
\newsiamthm{claim}{Claim}

\headers{Optimization via conformal Hamiltonian systems on manifolds}{M. Ghirardelli}

\title{Optimization via conformal Hamiltonian systems on manifolds\thanks{Submitted to the editors 31/07/2023.}}

\author{Marta Ghirardelli\thanks{Norwegian University of Science and Technology, NTNU Trondheim, Norway 
  (\email{marta.ghirardelli@ntnu.no}).}}

\usepackage{amsopn}

\ifpdf
\hypersetup{
  pdftitle={aRxiv_paper},
  pdfauthor={M. Ghirardelli}
}
\fi

\begin{document}

\maketitle

\begin{abstract}
In this work we propose a method to perform optimization on manifolds. We assume to have an objective function $f$ defined on a manifold and think of it as the potential energy of a mechanical system. By adding a momentum-dependent kinetic energy we define its Hamiltonian function, which allows us to write the corresponding Hamiltonian system. We make it conformal by introducing a dissipation term: the result is the continuous model of our scheme. We solve it via splitting methods (Lie-Trotter and leapfrog): we combine the RATTLE scheme, approximating the conserved flow, with the exact dissipated flow. The result is a conformal symplectic method for constant stepsizes. We also propose an adaptive stepsize version of it. We test it on an example, the minimization of a function defined on a sphere, and compare it with the usual gradient descent method.
\end{abstract}



\section{Introduction}
Over the last few decades, optimization methods for problems on manifolds have been gaining more and more popularity \cite{absil2009optimization}, \cite{boumal2023introduction}. In many cases, they can be seen as an alternative to solving constrained optimization problems in Euclidean spaces. More precisely, the manifold setting allows to include the constraints in a more natural geometric way. When considering applications to deep learning, constraints on the parameter space may sometimes be useful to prevent the problem of exploding or vanishing gradients, which often occurs e.g. during the training of recurrent neural networks. It would then be desirable to choose compact manifolds, such as the Stiefel manifold or the orthogonal group. 

Still in the field of machine learning and deep learning, Gradient Descent (GD) algorithms are widely used as they require only first order information of the function $f$ to be minimized, and they are easy to implement. They are hence suitable for minimizing e.g. cost or loss function in such contexts, where one has to deal with large amounts of data. Unfortunately, these methods are proved to achieve linear convergence only, as long as $f$ is smooth and strongly convex, and can consequently be rather slow. Accelerated versions have been studied, especially in the Euclidean setting where we find Nesterov's accelerated gradient, Polyaks's heavy ball and a relativistic gradient descent method. The last two of these, in particular, can be seen as the discretization of some conformal Hamiltonian systems \cite{francca2020conformal} where $f$ plays the role of potential energy.
In the Euclidean case, conformal Hamiltonian systems have also been studied and used to formulate optimization schemes e.g. in \cite{francca2020conformal}, or in \cite{maddison2018hamiltonian} where the authors present methods that converge linearly, using first order gradient information and constant stepsize, on a broader class of functions than the smooth and strongly convex one. In the manifold setting, such systems have been used e.g. in \cite{garcia2020structure} but not with the aim of performing optimization. Other numerical methods that do perform optimization on manifolds can be found in literature, for example \cite{zhang2018towards, celledoni2020energy, celledoni2018dissipative}.

We here present a generalization of the approach of \cite{francca2020conformal} to manifolds. This entails the use of conformal Hamiltonian systems and corresponding structure preserving numerical schemes. We start with an objective function $f$ defined on a manifold $M$ and introduce the Hamiltonian function by adding a momentum dependent kinetic energy to the potential energy $f$. The related Hamiltonian vector field is defined on the cotangent bundle of $M$, namely $T^*M$, where an exact natural symplectic form $\omega$ can be defined. The manifold $(T^*M, \omega)$ is then called exact, and the conformal Hamiltonian system is recovered by adding to the conservative vector field, a dissipative one, the Liouville vector field \cite{mclachlan2001conformal}. The flow of this vector field can then be approximated via splitting methods (SM), composing the conservative flow with the dissipative one. 

In this work, we use the RATTLE scheme \cite{andersen1983rattle,3-540-30663-3}, to approximate the conservative part of the flow. For the dissipative part, the exact solution can be computed. The SM of both order one and two can then be obtained and we prove some properties they satisfy. In particular, we prove that both of them are $c$-conformal symplectic, the parameter $c$ being the same. In addition, we formulate an adaptive stepsize version following the idea in \cite{wadia2021optimization}. Numerical results are shown for a particular example, the minimization of a function defined on a sphere, for which the proposed method is compared with a standard GD scheme. We perform a somewhat detailed numerical study, aiming at optimizing the parameters for both GD and the proposed method: at least in this example, our conformal symplectic scheme appears to be a competitive alternative to the usual GD.


\section{Optimization through descent algorithm in the Euclidean setting}
\label{Optimization through descent algorithm in the Euclidean setting}
We consider the unconstrained minimization problem
\begin{equation}
\label{min pb}
    \min_{x\in\mathbb{R}^d} f(x),
\end{equation}
where $f:\mathbb{R}^d \to \mathbb{R}$ is a $\mathcal{C}^2$ differentiable convex function. In the context of machine learning or deep learning, one could think of $f$ as a cost or loss function to minimize. 

\subsection{Gradient descent (GD)}
An important family of optimization methods used to solve (\ref{min pb}) consists of the so called Gradient Descent methods (GD). Such schemes only require the knowledge of the first order derivatives of $f$, i.e. $\nabla f(x)$, and are widely used as they are easy to implement. They are proved to achieve linear convergence using constant stepsizes on functions $f$ which are smooth and strongly convex \cite{maddison2018hamiltonian}.

\subsubsection{Accelerated descent methods}
Basic gradient descent algorithms are computationally efficient but can be rather slow. For this reason, \textit{accelerated gradient methods} were developed, able to achieve best worst-case complexity bounds \cite{francca2020conformal}. Among them, two of the most well known are Polyak's heavy ball method \cite{polyak1964some} and Nesterov's accelerated gradient method\cite{Nesterov1983AMF}.

\subsubsection*{Polyak's heavy ball method}
The heavy ball scheme reads 
\begin{equation}
    \label{CM}
    p_{k+1} = \mu p_k - \epsilon \nabla f(q_k), \qquad q_{k+1} = q_k + p_{k+1},
\end{equation}
where $k=0,1,\dots$ is the iteration number, $\mu\in(0,1)$ the momentum factor, $\epsilon >0$ the learning rate and $f:\mathbb{R}^d\to\mathbb{R}$ the usual function to be minimized. This scheme is also known as the classical momentum method (CM) \cite{francca2020conformal}.

\subsubsection*{Nesterov's accelerated gradient method}
A possible formulation of Nesterov's accelerated gradient scheme reads 
\begin{equation}
    \label{NAG}
    p_{k+1} = \mu p_k - \epsilon \nabla f(q_k + \mu p_k), \qquad q_{k+1} = q_k + p_{k+1},
\end{equation}
with $k$, $\mu$, $\epsilon$ and $f$ as for CM \cite{francca2020conformal}.

\subsection{Conformal Hamiltonian systems in Euclidean space}
Considering the Euclidean space $\mathbb{R}^{2d}$ with coordinates $(q,p)$ and the standard symplectic form $\omega = \text{d}q \wedge \text{d}p$, a \textit{conformal Hamiltonian system} has the form
\begin{equation}
    \label{conformal VF Euclidean}
    \Dot{q} = \frac{\partial H}{\partial p}, \qquad \Dot{p} = - \frac{\partial H}{\partial q} + cp,
\end{equation}
where $H:\mathbb{R}^{2d}\to\mathbb{R}$ is the Hamiltonian. We consider Hamiltonians of the form 
\begin{equation*}
    H(q,p) = \frac{1}{2}\langle p,p\rangle + f(q)
\end{equation*}
where $f$ is a $\mathcal{C}^2$ convex function representing the potential energy. The property of such vector fields is that their flow $\varphi$ is \textit{conformal}, meaning that it satisfies
\begin{equation}
    \label{conformal flow Euclidean}
    \varphi^*\omega = e^{ct}\omega,
\end{equation}
hence the symplectic form contracts (expands) exponentially along the flow if $c<0$ $(c>0)$. In contrast to the classical Hamiltonian case, the energy is not preserved as long as $c\ne 0$. In particular, when $c<0$ we have a dissipation since 
\begin{equation*}
    \Dot{H}(q,p) = \langle \nabla_q H, \Dot{q} \rangle +  \langle \nabla_p H, \Dot{p} \rangle = c \langle p,p \rangle \le 0.
\end{equation*}
As a consequence, $H$ is a Lyapunov function for the system, which is driven to the stationary point $(q_*,0)$ where $q_*$ is a minimum of $f$  \cite{maddison2018hamiltonian}.

The CM method can be seen as the discretization of a conformal Hamiltonian system \cite{francca2020conformal}. With $H(q,p) = \frac{1}{2}\langle p,p\rangle + f(q)$, CM is in fact a reformulation of the system
\begin{equation}
    \begin{split}
        \Dot{q} &= \nabla_p H = p,\\
        \Dot{p} &= -\nabla_q H - \gamma p = - \nabla f(q) -\gamma p.
    \end{split}
\end{equation}
The parameter $\gamma >0$ is the \textit{damping} constant which controls the dissipation, and is related with the \textit{momentum} factor $\mu$ in (\ref{CM}) via \begin{equation*}
    \mu = e^{-\gamma h}.
\end{equation*}

One could also come up with other numerical schemes by replacing the usual kinetic energy with another momentum-dependant one $T(p)$ and discretizing the corresponding system. As an example \cite{francca2020conformal}, one could consider the \textit{relativistic} kinetic energy of a particle of mass $m$
\begin{equation*}
    T(p) = v_c \sqrt{||p||^2 + (m v_c)^2},
\end{equation*}
which takes into account the limiting speed of light $v_c$. The system would then be
\begin{equation*}
    \Dot{q} = \frac{v_c p}{\sqrt{||p||^2 + (mv_c)^2}}, \qquad
    \Dot{p} = \nabla f(q) - \gamma p,
\end{equation*}
and the considerations above would still hold as, for $c<0$,
\begin{equation*}
    \Dot{H}(q,p) = \frac{v_c c}{\sqrt{||p||^2 + (mv_c)^2}} \langle p,p \rangle \le 0.
\end{equation*}
Such a choice is clever when one has to deal with large gradients: here $p$ is normalized, $\Dot{q}$ remains bounded even if $p$ does not and divergence is controlled.

On the other hand, NAG is not the discretization of a conformal Hamiltonian system \cite{francca2020conformal}. The continuous equation of this scheme is a second order non-autonomous ODE \cite{su2015differential} of the form 
\begin{equation}
    \Ddot{x} + \frac{3}{t} \Dot{x} + \nabla f(x) = 0.
\end{equation}

\section{Conformal Hamiltonian systems on a manifold}
\label{Conformal Hamiltonian systems on a manifold}
Since the purpose of this work is to perform optimization on manifolds, in this section we give some definitions and results related to conformal Hamiltonian systems on a manifold (for further details we refer to \cite{mclachlan2001conformal}). Their discretization will then lead to optimization schemes on a manifold.\\

\textit{Conformal symplectic vector fields} are those which preserve a symplectic form $\omega$ up to a constant. \textit{Symplectic diffeomorphisms} on a manifold form a group. According to the Cartan classification\footnote{The Cartan classification only applies to \textit{transitive}, \textit{primitive} $\infty$-dimensional Lie pseudogroups $G$ on complex manifolds $M$ where:
\begin{itemize}
    \item \textit{transitive}: for all $x,y \in M$ exists $\phi\in G$ such that $\phi(x)=y$;
    \item \textit{primitive}: there is no foliation of $M$ which is left invariant by every element of $G$ (a Poisson manifold with its automorphisms is non-primitive);
    \item \textit{Lie} because they are defined as solution of PDEs;
    \item \textit{pseudo} because the solutions are only \textit{local} diffeomorphisms, so their composition is defined only when domains and ranges overlap.
\end{itemize} }, there exist six classes of groups of diffeomorphisms. Among them, two are called \textit{conformal groups} and we are interested in Diff$_\omega^c$ i.e. the diffeomorphisms preserving $\omega$ up to a constant.

\subsection{Conformal vector fields on exact manifolds}
Consider a symplectic manifold $(M, \omega)$\footnote{$\omega$ is a closed (i.e. $\mathbf{d}\omega = 0$) and non-degenerate 2-form.}. If $M$ is \textit{exact} then we can write $\omega = -\textbf{d}\theta$. \\
Let $H\in\mathcal{C}^\infty (M)$ and $X_H$ the corresponding Hamiltonian vector field. We denote
with $\mathbf{L}_{X} \omega$ the Lie derivative
\begin{equation*}
    \mathbf{L}_{X} \omega = \mathbf{i}_{X} \mathbf{d} \omega + \mathbf{d} \mathbf{i}_{X} \omega,
\end{equation*}
where $X$ is a vector field and $\mathbf{i}_{X}\omega$ the interior product,
\begin{equation*}
    \mathbf{i}_{X}\omega  = \omega(X,\cdot).
\end{equation*} 
\begin{definition}
    The vector field $X^c$ is \textit{conformal} with parameter $c\in\mathbb{R}$ if
    \begin{equation}
        \label{conformal VF}
        \mathbf{L}_{X^c} \omega = c\,\omega.
    \end{equation}
    The diffeomorphism $\varphi^c$ is \textit{conformal} with parameter $c\in\mathbb{R}$ if
    \begin{equation}
        \label{conformal diffeo}
        (\varphi^c)^* \omega = c\,\omega.
    \end{equation}
\end{definition}
The diffeomorphisms $\varphi^c$ form the pseudogroup Diff$_\omega^c$.
\begin{prop}[Proposition 1 in \cite{mclachlan2001conformal}]
    The followings hold:
    \begin{enumerate}
        \item The time-t flow of $X^c$ is conformal with parameter $e^{ct}$.
        \item $(M,\omega)$ admits a conformal vector field with parameter $c\ne0$ iff $M$ is exact.
        \item Given $H\in\mathcal{C}^\infty(M)$ and $M$ exact, the vector field $X_H^c$ defined by
        \begin{equation}
            \mathbf{i}_{X^c_H}\omega = \mathbf{d}H - c\,\theta
        \end{equation}
        is conformal.
        \item If, in addition, $H^1(M)=0$\footnote{i.e. every closed form is exact and the \textit{De Rham Cohomology} of order 1, $H^1$, is simply connected.}, then given $X^c$, there exists a function $H$ such that $X^c = X_H^c$, and the set of conformal vector fields on $M$ is given by
        \begin{equation}
            \{X_H + cZ \, : \, H\in\mathcal{C}^\infty(M)\},
        \end{equation}
        where $Z$ is the Liouville vector field \cite{libermann2012symplectic} defined by
        \begin{equation}
        \label{Liouville VF}
            \textbf{d}_Z \omega = -\theta.
        \end{equation}
    \end{enumerate}
\end{prop}
\begin{example}
    Euclidean setting. Consider $\big(M=\mathbb{R}^{2n}, \omega=dq \wedge dp \big)$ with $(q,p)$ local coordinates. Let $\xi,\eta \in \mathbb{R}^{2n}$, then
    \begin{equation*}
        (dq \wedge dp) (\xi,\eta) = dq(\xi) \, dp(\eta)-dp(\xi)\, dq(\eta) = 
        \xi^T \begin{pmatrix} 0 & I \\ -I & 0 \end{pmatrix} \eta = \xi^T \Omega \eta.
    \end{equation*} 
    The 2-form $\omega$ is closed since $d\omega = 0$. We need it to be exact so that $M$ is an exact manifold. We need a 1-form $\theta$ such that $\omega = -d\theta.$ Once there, we can find the Liouville vector field $Z=(Z_q, Z_p)$ defined by $\mathbf{i}_Z \omega = \theta$, i.e.
    \begin{equation*}
        \mathbf{i}_Z \omega = \omega(Z, \cdot) = Z^T \Omega \, (\cdot) = (-Z_p^T, Z_q^T) \, (\cdot) = \theta \, (\cdot).
    \end{equation*}
    We have multiple possibilities, e.g.:
    \begin{enumerate}
        \item $\theta = p \, dq$, then
        \begin{equation*}
            -Z_p^T \eta_q + Z_q^T \eta_p = p^T \eta_q
            \qquad \Rightarrow \qquad Z = \begin{pmatrix}
                0 \\ -p
            \end{pmatrix}.
        \end{equation*}
        \item $\theta = -q \, dp$, then
        \begin{equation*}
            -Z_p^T \eta_q + Z_q^T \eta_p =-q^T \eta_p
            \qquad \Rightarrow \qquad Z = \begin{pmatrix}
                -q \\ 0
            \end{pmatrix}.
        \end{equation*}
    \end{enumerate}
    One can then recover all (provided $H^1(M)=0$) the conformal vector fields on $M$ as they are given by
    \begin{equation*}
        X^c = X_H + c\, Z, \qquad H \text{ Hamiltonian, } c\in\mathbb{R}.
    \end{equation*}
\end{example}
\begin{example}
    2-sphere. Consider $\big(M=S^2, \omega_p =x_p\, dy \wedge dz + y_p\, dz \wedge dx +z_p\, dx \wedge dy \big)$ with $S^2 = \{(x,y,z)\in\mathbb{R}^3 \, :\, x^2+y^2+z^2=1\}$ and  $(x,y,z)$ local coordinates. \\
    Passing to the cylindrical coordinates with $\theta \in [0,2\pi)$, $z\in[-1,1]$ we have
    \begin{align*}
        x &= \sqrt{1-z^2} \cos \theta & \Rightarrow \qquad
        & dx = -\frac{z}{\sqrt{1-z^2}} \cos \theta \, dz - \sqrt{1-z^2} \sin\theta \, d\theta, \\
        y &= \sqrt{1-z^2} \sin \theta &\Rightarrow \qquad
        & dy = -\frac{z}{\sqrt{1-z^2}} \sin \theta \, dz + \sqrt{1-z^2} \cos\theta \, d\theta, \\[.9ex]
        z &= z & \Rightarrow \qquad & dz = dz.
    \end{align*}
    The 2-form $\omega$ becomes
    \begin{equation*}
        \omega = d\theta \wedge dz = -d(-\theta \, dz) =: -d\Theta_1, \qquad \text{or} \qquad \dots = -d(z \, d\theta) =: -d\Theta_2,
    \end{equation*}
    and the corresponding Liouville vector fields are
    \begin{align*}
            \omega(Z,\cdot) &= d\theta(Z)\,dz(\cdot) - dz(Z) \,d\theta(\cdot) = -\theta \, dz(\cdot) & \Rightarrow \qquad Z = \begin{pmatrix}
                -\theta \\ 0
            \end{pmatrix}, \\
            \omega(Z,\cdot) &= d\theta(Z)\,dz(\cdot) - dz(Z) \,d\theta(\cdot) = z \, d\theta(\cdot) &\Rightarrow \qquad Z = \begin{pmatrix}
                0 \\ -z
            \end{pmatrix}.
    \end{align*}
    The conformal vector fields on $M$ (provided $H^1(M)=0$) are then given by
    \begin{equation*}
        X^c = X_H + c\, Z, \qquad H \text{ Hamiltonian, } c\in\mathbb{R}.
    \end{equation*}
\end{example}
In the above examples we could choose local coordinates so that the Liouville vector field $Z$ is linear and depends only on one of the variables, and the 2-form $\omega$ results to be the standard canonical one. We recall Darboux' theorem.
\begin{theorem}[Darboux, \cite{marsden2013introduction} p. 148]
    Let $\big(M,\omega \big)$ be a symplectic $2n$-dimensional manifold and $p$ a point on $M$. Then, there exists a coordinate chart \\$\big(\mathcal{U}, x_1,\dots,x_n,y_1,\dots,y_n \big)$ centered at $p$ such that, on $\mathcal{U}$
    \begin{equation*}
        \omega|_\mathcal{U} = \sum_{j=1}^n dx_j \wedge dy_j,
    \end{equation*}
    i.e. $\omega$ is locally the standard canonical 2-form.
\end{theorem}
An important class of symplectic manifold is that of \textit{cotangent bundles} of smooth manifolds. Given $M$ smooth manifold, we define the \textit{Liouville 1-form} $\theta$ on its cotangent bundle $T^*M$ 
\begin{equation*}
    \theta: T^*M \to T^*(T^*M), \qquad \theta(\alpha)(X_\alpha) := \langle \alpha, T\pi (X_\alpha) \rangle,
\end{equation*}
where $\alpha\in T^*M$, $X_\alpha \in \mathfrak{X}(T^*_\alpha M)$ and $T\pi$ is the tangent lift of the projection $\pi:T^*M \to M$. In local coordinates $(\mathcal{U}, q_1,\dots,q_n,p_1,\dots,p_n)$ on $T^*M$
\begin{equation*}
    \theta|_\mathcal{U} = \sum_{i=1}^n p_i dq_i, \qquad \text{and} \qquad d\theta|_\mathcal{U} = \sum_{i=1}^n dq_i \wedge dp_i,
\end{equation*}
i.e. $d\theta$ locally is the standard symplectic form.\\
Locally, the canonical symplectic form can always be obtained by choosing Darboux coordinates.

\section{Optimization via conformal Hamiltonian systems on a manifold}
The goal is to minimize a $\mathcal{C}^2$ convex real-valued function $f$ that is defined on a manifold $M$
\begin{equation}
\label{pb}
    \min_{q\in M} f(q), \qquad f:M\to \mathbb{R}.
\end{equation}
To solve such optimization problem one can either
\begin{enumerate}
    \item work in the manifold setting using local coordinates and relying on the intrinsic properties of the manifold only, \cite{absil2009optimization}, \cite{boumal2023introduction}, or
    \item consider the submersion of $M$ in an Euclidean space and solve a constrained minimization problem.
\end{enumerate}
The idea is then to let $f$ be the potential energy of a mechanical system, whose kinetic energy is described by some function $T(\Dot{q})$.

In this work, we choose to follow the second direction. In an Euclidean space, the constraints can be enforced via Lagrange multipliers. We make the following assumptions:
\begin{itemize}
    \item $M$ is embedded in $\mathbb{R}^n$ and there we have coordinates $q_1, \dots, q_n$;
    \item $g(q) = 0$, $g:\mathbb{R}^n \to \mathbb{R}^m$ with $m<n$ are the constraints;
    \item $T(\Dot{q}) = \frac{1}{2} \Dot{q}^T M^{-1} \Dot{q}$, $f(q)$ are kinetic and potential energy respectively, $M$ being the mass matrix.
\end{itemize}
The augmented Lagrange function is \cite{3-540-30663-3}
\begin{equation}
    L(q,\Dot{q}) = T(\Dot{q}) - f(q) - g(q)^T \lambda,
\end{equation}
where $\lambda = (\lambda_1, \dots, \lambda_m)$ are the Lagrange multipliers. The Euler-Lagrange equations of the variational formulation problem for $\int_0^t L(q,\Dot{q}) dt$ are
\begin{equation}
    \frac{d}{dt} \bigg(\frac{\partial L}{\partial \Dot{q}}\bigg) - \frac{\partial L}{\partial q} = 0.
\end{equation}
They can be rewritten as a system of ODE, together with the constraints
\begin{equation}
    \begin{cases}
        \Dot{q} = v, \\
        \Dot{v} = \dots, \\
        g(q) = 0.
    \end{cases}
\end{equation}

\subsection{Formulation of the conformal Hamiltonian system}
As the system is of mechanical type, we consider the Hamiltonian formulation rather than the Lagrangian one. In what follows we sketch the derivation of it and we refer to \cite{3-540-30663-3} for further details.

\subsubsection{Conservative system}
We use the momentum coordinates
\begin{equation*}
    p = \frac{\partial L}{\partial \Dot{q}}
\end{equation*}
in place of $v=\Dot{q}$ and write the Hamiltonian as the total energy of the system:
\begin{equation}
    H(q,p) = T(p) + f(q).
\end{equation}
The system becomes
\begin{equation}
\label{constrained Hamiltonian system}
    \begin{split}
        \Dot{q} &= H_p(q,p), \\
        \Dot{p} &= - H_q(q,p) - G(q)^T \lambda, \\
        0 &= g(q),
    \end{split}
\end{equation}
where $G(q)$ is the Jacobian of $g$. Differentiating the constraint once we get
\begin{equation}
    0 = G(q) H_p(q,p),
\end{equation}
and twice
\begin{equation}
    0 = \frac{\partial}{\partial q}  \Big(G(q)H_p(q,p) \Big) H_p(q,p) - G(q) H_{pp}(q,p) \Big(H_q(q,p)+G(q)^T\lambda \Big).
\end{equation}
Provided that 
\begin{equation}
    \label{lambda}
    G(q)H_{pp}(q,p)G(q)^T \qquad \text{is invertible,}
\end{equation}
one can recover an expression for $\lambda$ in terms of the variables $(q,p)$.
Inserting it into (\ref{constrained Hamiltonian system}) gives a differential equation
for $(q,p)$ on the manifold
\begin{equation}
\label{manifold}
    \mathcal{M} = \{(q,p) \, : \, g(q) = 0, \, G(q)H_p(q,p) = 0 \},
\end{equation}
i.e. the cotangent bundle $T^*M$ of the configuration manifold $M$. \\
The system (\ref{constrained Hamiltonian system}) satisfies the following properties:
\begin{enumerate}
    \item \textbf{Preservation of the Hamiltonian.} Differentiating the Hamiltonian $H$ along the solution $(q(t),p(t))$ of the system (\ref{constrained Hamiltonian system}) we get
    \begin{equation*}
        -H_p^T H_q - H_p^T G^T \lambda + H_q^T H_p = - H_p^T G^T \lambda = 0,
    \end{equation*}
    as we are on $\mathcal{M}$.
    \item \textbf{Symplecticity of the flow.} The flow $\varphi$ of the system (\ref{constrained Hamiltonian system}) is a transformation on $\mathcal{M}$, $\varphi:\mathcal{M}\to\mathcal{M}$, so its derivative is a mapping between the corresponding tangent spaces:
    \begin{equation*}
        \frac{d}{dt}\varphi : T\mathcal{M} \to T\mathcal{M}.
    \end{equation*}
    The map $\varphi$ is symplectic if, for every $x=(q,p)\in\mathcal{M}$,
    \begin{equation}
    \label{symplecticity condition}
        \xi_1^T \varphi'(x)^T J \varphi'(x)\xi_2 = \xi_1^T J \xi_2, \qquad \mbox{for all } \xi_1, \xi_2 \in T_x\mathcal{M}.
    \end{equation}
    In the expression above, $\varphi'(x)$ is the directional derivative:
    \begin{equation*}
        \varphi'(x)\xi := \frac{d}{d\tau}\bigg|_{\tau=0} \varphi(\gamma(\tau))
    \end{equation*}
    for $\xi\in T_x\mathcal{M}$ and $\gamma$ is a path on $M$ such that
    \begin{equation*}
        \gamma(0) = x, \qquad \Dot{\gamma}(0) = \xi.
    \end{equation*}
    In the case $\varphi$ is defined and continuously differentiable in an open set $U\subset \mathbb{R}^{2n}$ which contains $\mathcal{M}$, then $\varphi'$ is the usual Jacobian.
\end{enumerate}
\begin{theorem}[Theorem 1.2 in \cite{3-540-30663-3}, section VII]
    Let $H(q,p)$ and $g(q)$ be twice continuously differentiable. The flow $\varphi_t : \mathcal{M} \to \mathcal{M}$ of the system (\ref{constrained Hamiltonian system}) is then a symplectic transformation on $\mathcal{M}$.
\end{theorem}

\subsubsection{Dissipative system}
The system (\ref{constrained Hamiltonian system}) is conservative: its evolution occurs on level sets of $H$ but does not lead to a point where the energy is minimized \cite{3-540-30663-3}. In order to make that happen, we need to add a dissipation: by doing so, the evolution leads to a minimizer $(q^*, p^*)$ of $H$ \cite{maddison2018hamiltonian}. Moreover, if the Hamiltonian is separable, i.e.
\begin{equation*}
    H(q^*, p^*) = f(q^*) + T(p^*),
\end{equation*}
then $q^*$ is a minimizer of $f$, $p^*$ is a minimizer of $T$ and the original problem (\ref{pb}) is solved \cite{mclachlan2001conformal}. A dissipative version of (\ref{constrained Hamiltonian system}) is a conformal Hamiltonian system of the form (see section \ref{Conformal Hamiltonian systems on a manifold}) 
\begin{equation}
\label{dissipative constrained Hamiltonian system}
    \begin{split}
        \Dot{q} &= H_p(q,p), \\
        \Dot{p} &= - H_q(q,p) - G(q)^T \lambda - \gamma p, \\
        0 &= g(q),
    \end{split}
\end{equation}
where $\gamma >0$ is the parameter of dissipation.

\subsection{Numerical method}
After building the continuous model, we need to rely on some numerical scheme to solve it. A possibility is to split it into two parts, a conservative one and a dissipative one, compute the flow of each of them separately and combine them via splitting schemes.

\subsubsection{Conservative system}
The conservative system coincides with (\ref{constrained Hamiltonian system}). We call $\Phi_t^C:\mathcal{M}\to \mathcal{M}$ its flow and to find an approximation $\Psi_h^C$ of it we use the RATTLE numerical method \cite{andersen1983rattle}. Assuming $h$ is the time step and $(q_n,p_n)$ the approximation of the flow at $t_n$, the approximation at $t_{n+1}$ is found by solving the system:
\begin{equation}
    \label{RATTLE general}
    \begin{split}
        p_{n+1/2} &= p_n - \frac{h}{2} \big(H_q(p_{n+1}, q_n) + G(q_n)^T \lambda_n \big), \\
        q_{n+1} &= q_n + \frac{h}{2} \big(H_p(p_{n+1/2}, q_n) + H_p(p_{n+1/2}, q_{n+1} \big), \\[1ex]
        0 &= g(q_{n+1}), \\
        p_{n+1} &= p_{n+1/2} - \frac{h}{2} \big(H_q(p_{n+1}, q_{n+1}) + G(q_{n+1})^T \mu_n \big), \\[1ex]
        0 &= G(q_{n+1}) H_p (p_{n+1}, q_{n+1}).
    \end{split}
\end{equation}
This scheme ensures that $\Psi_h^C(q_n, p_n) = (q_{n+1},p_{n+1})$ lies on the manifold $\mathcal{M}$, i.e. $\Psi_h^C:\mathcal{M}\to \mathcal{M}$. In particular, the parameter $\lambda_n$ and the third equation ensure that the constraints are satisfied, so that $q_{n+1}$ is on $M$. The parameter $\mu_n$ and the last equation, instead, ensure that $p_{n+1}$ is on $TM$ (indeed $p_{n+1}$ is the projection of $p_{n+1/2}$ on $TM$). Existence and (local) uniqueness of the solution are discussed in \cite{3-540-30663-3}, as well as the proof of the following statement:
\begin{theorem}(Theorem 1.3 in \cite{3-540-30663-3}, section VII)
    The numerical scheme (\ref{RATTLE general}) is symmetric, symplectic and convergent of order two.
\end{theorem}
As long as the kinetic energy $T$ depends only on $p$, which is always the case for us, the Hamiltonian is separable and (\ref{RATTLE general}) simplifies to
\begin{equation}
    \label{RATTLE separable}
    \begin{split}
        p_{n+1/2} &= p_n - \frac{h}{2} \big(H_q(q_n) + G(q_n)^T \lambda_n \big), \\[1ex]
        q_{n+1} &= q_n + h H_p(p_{n+1/2}),\\[1.3ex]
        0 &= g(q_{n+1}), \\
        p_{n+1} &= p_{n+1/2} - \frac{h}{2} \big(H_q(q_{n+1}) + G(q_{n+1})^T \mu_n \big), \\[1ex]
        0 &= G(q_{n+1}) H_p (p_{n+1}).
    \end{split}
\end{equation}
Unfortunately, the system remains nonlinear and implicit as $g$ is nonlinear and $\lambda_n$ is unknown (the expression (\ref{lambda}) would require $(q_n, p_{n+1/2})$ to compute it). 

\subsubsection{Dissipative system}
The dissipative part of the system (\ref{constrained Hamiltonian system}) is
\begin{equation}
\label{dissipative part}
    \begin{split}
        \Dot{q} &= 0, \\
        \Dot{p} &= - \gamma p,
    \end{split}
\end{equation}
and can be solved exactly. Given $(q_0, p_0)\in\mathcal{M}$ initial condition, the flow $\Phi_t^D:\mathcal{M}\to \mathcal{M}$ at time $t$ is given by
\begin{equation*}
    \Phi_t^D (q_0, p_0) = \big(q_0, e^{-\gamma t}p_0\big).
\end{equation*}

\subsubsection{Splitting scheme}
The flow of the whole system (\ref{constrained Hamiltonian system}) is found by composing the conservative and dissipative ones, $\Phi_t^C$ and $\Phi_t^D$. We test
\begin{enumerate}
    \item The Lie-Trotter first order integration scheme
    \begin{equation*}
        \label{Lie-Trotter}
        \Psi_h^1 = \Psi_h^C \circ \Psi_h^D,
    \end{equation*}
    leading to the numerical method
    \begin{equation}
    \label{SM order 1}
        \begin{split}
        p_{n+1/2} &= e^{-\gamma h} p_n - \frac{h}{2} \big(H_q(q_n) + G(q_n)^T \lambda_n \big), \\[1ex]
        q_{n+1} &= q_n + h H_p(p_{n+1/2}), \\[1.3ex]
        0 &= g(q_{n+1}), \\
        p_{n+1} &= p_{n+1/2} - \frac{h}{2} \big(H_q(q_{n+1}) + G(q_{n+1})^T \mu_n \big), \\[1ex]
        0 &= G(q_{n+1}) H_p (p_{n+1}).
        \end{split}
    \end{equation}
    \item The leapfrog symmetric second order integration scheme
    \begin{equation*}
        \label{leapfrog}
        \Psi_h^2 = \Psi_{h/2}^C \circ \Psi_h^D \circ \Psi_{h/2}^C,
    \end{equation*}
    leading to the numerical method
    \begin{equation}
    \label{SM order 2}
        \begin{split}
        p_{n+1/2} &= e^{-\gamma h/2} p_n - \frac{h}{2} \big(H_q(q_n) + G(q_n)^T \lambda_n \big), \\[1ex]
        q_{n+1} &= q_n + h H_p(p_{n+1/2}), \\[1.3ex]
        0 &= g(q_{n+1}), \\
        p_{n+1} &= e^{-\gamma h/2} \Big( p_{n+1/2} - \frac{h}{2} \big(H_q(q_{n+1}) + G(q_{n+1})^T \mu_n \big) \Big), \\[1ex]
        0 &= G(q_{n+1}) H_p (p_{n+1}).
        \end{split}
    \end{equation}
\end{enumerate}
We now see some properties of the numerical schemes (\ref{SM order 1}) and (\ref{SM order 2}) whose proofs are modifications of the corresponding ones for the RATTLE method (\ref{RATTLE general}) in \cite{3-540-30663-3}, section VII.

\begin{theorem}
\label{properties SM 1}
    The numerical method (\ref{SM order 1}) has a unique solution, is conformal symplectic with parameter $e^{-\gamma h}$ and convergent of order 1.
\end{theorem}
\begin{proof}
    \textbf{Existence and uniqueness of the numerical solution.} First we prove existence and uniqueness of $(q_{n+1}, p_{n+1/2})$ using the first three lines of (\ref{SM order 1}). By inserting the expression of $q_{n+1}$ in $0=g(q_{n+1})$ we obtain a nonlinear system for $p_{n+1/2}$ and $h\lambda_n$. The implicit function theorem cannot be directly applied to prove existence and uniqueness of the solution because of the factor $h$ in front of $H(q_n, p_{n+1/2})$. We therefore rewrite the second of those equations as follows
    \begin{equation*}
        \begin{split}
        0 = g(q_{n+1}) & = g(q_n) + \int_0^1 G\big(q_n + \tau(q_{n+1}-q_n) \big) \, (q_{n+1}-q_n) \, d\tau \\
        & = \int_0^1 G\big(q_n + \tau h H_p(p_{n+1/2}) \big) \,  H_p(p_{n+1/2}) \, d\tau,
        \end{split}
    \end{equation*}
    where we have used $g(q_n)=0$, replaced the expression of $q_{n+1}$ and divided by $h$. The system becomes then $F(p_{n+1/2}, h\lambda_n, h) = 0$ with $F$ defined as
    \begin{equation*}
        F(p,\nu,h) = \begin{pmatrix}
            p - e^{\gamma h}p_n +\dfrac{1}{2} h H_q(q_n) + \displaystyle\frac{1}{2} G(q_n)^T \nu \\[1.5ex]
            \displaystyle \int_0^1 G\big(q_n + \tau h H_p(p) \big) \,  H_p(p) \, d\tau
        \end{pmatrix}.
    \end{equation*}
    We then have that
    \begin{equation*}
        F(p_n,0,0) = \begin{pmatrix}
            p_n - p_n + 0 \\[1ex]
            \displaystyle \int_0^1 G(q_n) \,  H_p(p_n) \, d\tau
        \end{pmatrix} = 0,
    \end{equation*}
    as $(q_n, p_n) \in \mathcal{M} = \{(q,p) \, : \, g(q) = 0, \, G(q)H_p(q,p) = 0 \}$ from (\ref{manifold}). Moreover, the Jacobian matrix
    \begin{equation*}
        \frac{\partial F}{\partial (p,\nu)} (p_n,0,0) = \begin{pmatrix}
            I & G(q_n)^T \\[.8ex]
            G(q_n)H_{pp}(p_n) & 0
        \end{pmatrix}
    \end{equation*}
    is invertible since $G(q(t))H_{pp}(p(t))G(q(t))^T$ is invertible by (\ref{lambda}). We can now apply the implicit function theorem which ensures existence and local uniqueness of the solution $(p_{n+1/2}, h\lambda_n)$ (and so also of $q_{n+1}$), provided $h$ is sufficiently small.\\
    The second step is to prove existence and uniqueness of the projection $p_{n+1}$ using the last two lines of (\ref{SM order 1}). Those equations form a nonlinear system for $(p_{n+1}, h\mu_n)$ and if we define
    \begin{equation*}
        \Tilde{F}(p,\nu,h) = \begin{pmatrix}
            p - p_{n+1/2} + \displaystyle \frac{h}{2} H_q(q_{n+1}) + G(q_{n+1})^T \nu \\[1.7ex]
            G(q_{n+1}) H_p (p)
        \end{pmatrix},
    \end{equation*}
    then $\Tilde{F}(p_{n+1/2}, 0, 0) = 0$ (since for $h=0$, $q_{n+1} = q_n$ and $p_{n+1/2} = p_n$). The Jacobian of $\tilde{F}$ with respect to $(p,\nu)$ evaluated in such point is the same as the one for $F$: as for the first step, existence and uniqueness of $(p_{n+1}, h\mu_n)$ follows by applying the implicit function theorem.\\
    \textbf{Conformal symplecticity.} To show (\ref{SM order 1}) is conformal symplectic with parameter $e^{-\gamma h}$ we need to verify a modification of (\ref{symplecticity condition}). In particular, we have to see that for every $x=(q,p)\in\mathcal{M}$,
    \begin{equation*}
        \xi_1^T \varphi'(x)^T J \varphi'(x)\xi_2 = e^{-\gamma h} \xi_1^T J \xi_2, \qquad \mbox{for all } \xi_1, \xi_2 \in T_x\mathcal{M}.
    \end{equation*}
    We verify this in two steps: first we study the mapping $\varphi_a : (q_n,p_n)\mapsto (q_{n+1},p_{n+1/2})$ defined by the first equations of (\ref{SM order 1}). We consider $\lambda_n$ as a function $\lambda(q_n,p_n)$ and differentiate $\varphi_a$ with respect to $(q_n,p_n)$:
    \begin{equation*}
        \varphi'_a = \begin{pmatrix}
            I & 0 \\[1ex]
            -\dfrac{h}{2} \big(S + G^T\lambda_q \big) & e^{-\gamma h} - \dfrac{h}{2} G^T\lambda_p 
        \end{pmatrix}
    \end{equation*}
    where $S:= H_{qq} + G_q^T\lambda$ is a symmetric matrix as both $H_{qq}$ and $G_q$ are Hessians. We then compute the product
    \begin{equation*}
        {\varphi'}_{a}^{T} J \varphi'_a = \begin{pmatrix}
            -\dfrac{h}{2} \big(G^T \lambda_q - \lambda_q G \big) & e^{-\gamma h} - \dfrac{h}{2} G^T \lambda_p \\[2ex]
            -e^{-\gamma h} + \dfrac{h}{2} \lambda_p G^T & 0
        \end{pmatrix},
    \end{equation*}
    and multiply this matrix from the left by $\xi_1$ and from the right by $\xi_2$, with $\xi_1,\xi_2 \in T_{q_n,p_n}\mathcal{M}$. By partitioning $\xi = (\xi_q,\xi_p)$ and using that $G(q_n)\xi_{q,i} = 0$ for $i=1,2$, we have
    \begin{equation*}
        \begin{split}
            \xi_1^T {\varphi'}_{a}^{T} J \varphi'_a \xi_2 = & -\frac{h}{2} \xi_{q,1}^T G^T \lambda_q \xi_{q,2} + \frac{h}{2} \xi_{q,1}^T \lambda_q^T G \xi_{q,2} + e^{-\gamma h} \xi_{q,1}^T \xi_{p,2} \\
            & - \frac{h}{2} \xi_{q,1}^T G^T \lambda_p \xi_{p,2} - e^{-\gamma h} \xi_{p,1}^T \xi_{q,2} +\frac{h}{2} \xi_{p,1}^T \lambda_p^T G\xi_{q,2} \\[1ex]
            = & \,\,\, e^{-\gamma h} \xi_{q,1}^T \xi_{p,2} - e^{-\gamma h} \xi_{p,1}^T \xi_{q,2} \\[1.5ex]
            = & \,\,\, e^{-\gamma h} \xi_1^T J \xi_2,
        \end{split}
    \end{equation*}
    meaning that the mapping $\varphi_a$ is conformal symplectic with parameter $e^{-\gamma h}$. Then we study the projection step $\varphi_b : (q_{n+1},p_{n+1/2})\mapsto (q_{n+1},p_{n+1})$. We consider $\mu_n$ as a function $\lambda(q_{n+1},p_{n+1/2})$ and differentiate $\varphi_a$ with respect to $(q_{n+1},p_{n+1/2})$:
    \begin{equation*}
        \varphi'_a = \begin{pmatrix}
            I & 0 \\[1ex]
            -\dfrac{h}{2} \big(\Tilde{S} + G^T\mu_q \big) & I - \dfrac{h}{2} G^T\mu_p 
        \end{pmatrix}
    \end{equation*}
    where $\Tilde{S}:= H_{qq} + G_q^T\mu$ is a symmetric matrix. We proceed in the same way as before: we compute
    \begin{equation*}
        {\varphi'}_{b}^{T} J \varphi'_b = \begin{pmatrix}
            -\dfrac{h}{2} \big(G^T \mu_q - \mu_q G \big) & I - \dfrac{h}{2} G^T \mu_p \\[2ex]
            I + \dfrac{h}{2} \mu_p G^T & 0
        \end{pmatrix},
    \end{equation*}
    and multiply it from the left and from the right by $\xi_1$ and $\xi_2$ respectively, with $\xi_1,\xi_2 \in T_{q_{n+1},p_{n+1/2}}\mathcal{M}$. We get
    \begin{equation*}
        \xi_1^T {\varphi'}_{b}^{T} J \varphi'_b \xi_2 = \xi_1^T J \xi_2,
    \end{equation*}
    meaning that $\varphi_b$ is conformal symplectic with parameter 1, i.e. symplectic. As a composition of two conformal symplectic transformation with parameters $e^{-\gamma h}$ and 1, the numerical flow defined by (\ref{SM order 1}) is $e^{-\gamma h}$-conformal symplectic.\\
    \textbf{Convergence of order 1.} By Taylor-expanding the continuous equations in the system (\ref{dissipative constrained Hamiltonian system}) and assuming $(q(t),p(t))$ is the solution of such system passing through $(q_n,p_n) \in \mathcal{M}$ at time $t_n$, we obtain
    \begin{equation*}
        \begin{split}
            q(t_{n+1}) &= q_n + h H_p(p_n) + \mathcal{O}(h^2), \\
            p(t_{n+1}) &= p_n - hH_q(q_n) - hG(q_n)^T\lambda(q_n,p_n) + \mathcal{O}(h^2).
        \end{split}
    \end{equation*}
    We want to compare the above equations with the corresponding ones given by the numerical scheme (\ref{SM order 1})
    \begin{equation*}
        \begin{split}
            q_{n+1} &= q_n + hH_p(p_{n+1/2}), \\
            p_{n+1} &= p_n - \frac{h}{2} H_q(q_n) - \frac{h}{2} G(q_n)^T\lambda_n - \frac{h}{2}H_q(q_{n+1}) - \frac{h}{2}G(q_{n+1})^T\mu_n.
        \end{split}
    \end{equation*}
    By considering the estimates
    \begin{equation*}
        p_{n+1/2} = p_n + \mathcal{O}(h), \quad h\lambda_n = \mathcal{O}(h), 
        \quad p_{n+1} = p_{n+1/2} + \mathcal{O}(h), \quad h\mu_n = \mathcal{O}(h),
    \end{equation*}
    given by the application of the implicit function theorem to prove existence and uniqueness of the numerical solution, we get
    \begin{equation*}
        q_{n+1} = q(t_{n+1}) + \mathcal{O}(h^2), \qquad 
        p_{n+1} = p(t_{n+1}) + G(q(t_{n+1}))^T \nu + \mathcal{O}(h^2),
    \end{equation*}
    where $\nu:= \frac{h}{2}(\lambda_n - \mu_n)$. Inserting these relations in the last equation of (\ref{SM order 1}) leads to
    \begin{equation*}
        \begin{split}
            0 =&\,\, G(q_{n+1}) H_p(q_{n+1},p_{n+1}) \\
            =&\,\, G(q(t_{n+1})) H_p(q(t_{n+1}),p(t_{n+1})) \\
            & + G(q(t_{n+1})) H_{pp}(q(t_{n+1}),p(t_{n+1}))G(q(t_{n+1}))^T \nu + \mathcal{O}(h^2).
        \end{split}
    \end{equation*}
    The first term is 0 as we are on $\mathcal{M}$ and the matrix $GH_{pp}G^T$ is invertible by (\ref{lambda}). Hence $\nu = \mathcal{O}(h^2)$ and 
    \begin{equation}
    \label{local error}
        q_{n+1} = q(t_{n+1}) + \mathcal{O}(h^2), \qquad 
        p_{n+1} = p(t_{n+1}) +  \mathcal{O}(h^2),
    \end{equation}
    is the local error of the method. We now have to compute the global error summing up the propagated local ones. We do this for the variable $q$, then the calculations for $p$ are identical and lead to the same result. Assuming $q(t)$ is the solution starting at $q_0$ in $t_0$, we want to know the global error in $t_{n+1}$, i.e. $e_{n+1} := q(t_{n+1})-q_{n+1}$. We have
    \begin{equation*}
        \begin{split}
            q(t_{n+1}) & = q(t_n) + hH_p(p(t_n)) + Dh^2, \\
            q_{n+1} & = q_n + hH_p(p_n)
        \end{split}
    \end{equation*}
    where $Dh^2$ is the local error from (\ref{local error}), for some $D\in\mathbb{R}$. Computing the difference of the above equations gives
    \begin{equation*}
        \begin{split}
            e_{n+1} & = e_n + h \big(H_p(p(t_n)) - H_p(p_n) \big) + Dh^2 \\[1ex]
            & = e_n + h H_p(\overline{p}) \big(p(t_n)-p_n\big) + Dh^2
            = \big(1+h H_p(\overline{p})\big) e_n +Dh^2
        \end{split}
    \end{equation*}
    for some $\overline{p}\in[p_n, p(t_n)]$, by the mean value theorem. Assuming there exists $L\in\mathbb{R}$ such that $||H_{pp}(p)|| \le L$ for all $p$, then
    \begin{equation*}
        |e_{n+1}| \le (1+hL) |e_n| + Dh^2.
    \end{equation*}
    As the error in $t_0$ is 0, we can compute
    \begin{equation*}
        \begin{split}
            e_0 & = 0, \\
            e_1 & = Dh^2, \\
            e_2 & = (1+hL) Dh^2 + Dh^2,\\
            e_3 & = (1+hL) \big((1+hL) Dh^2 + Dh^2 \big) + Dh^2,\\
            \vdots \\
            e_N & = \sum_{i=0}^{N-1} (1+hL)^i Dh^2 = \frac{(1+hL)^N -1}{1+hL-1} Dh^2 = \frac{(1+hL)^N -1}{L} Dh.
        \end{split}
    \end{equation*}
    By Taylor expansion, $1+hL \le e^{hL}$, and if we do a finite number of steps $N$, then $hN<\infty$. Therefore
    \begin{equation*}
        |e_n| \le \frac{e^{hLN}-1}{L} Dh \xrightarrow{h\to 0} 0.
    \end{equation*}
    and the method is convergent of order 1.
\end{proof}

\begin{theorem}
\label{properties SM 2}
    The numerical method (\ref{SM order 2}) has a unique solution, is symmetric, conformal symplectic with parameter $e^{-\gamma h}$ and convergent of order 2.
\end{theorem}
\begin{proof}
    \textbf{Existence and uniqueness of the numerical solution.} The proof is very similar to the one in Theorem \ref{properties SM 1} for the SM of order 1. \\
    \textbf{Symmetry.} To prove the symmetry of the method (\ref{SM order 2}) we make a step back in time from $(q_{n+1}, p_{n+1})$, i.e. we exchange
    \begin{equation*}
        h \leftrightarrow -h, \qquad p_{n+1} \leftrightarrow p_n \qquad q_{n+1} \leftrightarrow q_n \qquad \lambda_{n} \leftrightarrow \mu_n.
    \end{equation*}
    The system becomes
    \begin{equation*}
        \begin{split}
        p_{n+1/2} &= e^{+\gamma h/2} p_{n+1} + \frac{h}{2} \big(H_q(q_{n+1}) + G(q_{n+1})^T \mu_n \big), \\[1ex]
        q_{n} &= q_{n+1} + h H_p(p_{n+1/2}), \\[1.3ex]
        0 &= g(q_{n}), \\
        p_{n} &= e^{+\gamma h/2} \Big( p_{n+1/2} + \frac{h}{2} \big(H_q(q_{n}) + G(q_{n})^T \lambda_n \big) \Big), \\[1ex]
        0 &= G(q_{n}) H_p (p_{n}).
        \end{split}
    \end{equation*}
    The third and the last equations are satisfied thanks to the consistency conditions of the initial value $(q_n, p_n)$. The second one is already identical to the second line of (\ref{SM order 2}), whereas the first and the fourth are just a rewriting of the fourth and the first in (\ref{SM order 2}), respectively.\\
    \textbf{Conformal symplecticity.} The proof is very similar to the one in Theorem \ref{properties SM 1} for the SM of order 1. The only difference is that here, the mapping $\varphi_a$ and $\varphi_b$ are both $e^{-\gamma h/2}$-conformal symplectic. Their composition, though, gives again a conformal symplectic transformation with parameter $e^{-\gamma h/2}e^{-\gamma h/2}=e^{-\gamma h}$.\\
    \textbf{Convergence of order 2.} Convergence of order 1 can be proved in the same way as in \ref{properties SM 1}. Convergence of order 2 follows by the symmetry of the method, as the order of a symmetric method is always even.
\end{proof}

\subsubsection{Adaptive scheme}
Instead of keeping the stepsize $h$ constant, we consider the adaptive stepsize routine proposed in \cite{wadia2021optimization} that is suitable for optimization problems. It is based on a proportional (P) control which sets the size of $h$ according to two parameters to be chosen in advance. In particular, assume $h_n$ and $x_{n} = (q_{n}, p_{n})$ are the $n^{th}$ stepsize and approximated solution respectively. Then we compute the next iterate with the first order integrator
\begin{equation*}
    x_{n+1}^1 = \Psi_{h_n}^1(q_{n}, p_{n}),
\end{equation*}
and compare it with the one obtained via the second order one
\begin{equation*}
    x_{n+1}^2 = \Psi_{h_n}^2(q_{n}, p_{n}),
\end{equation*}
to get
\begin{equation*}
    \delta_n(x_n, h_n) = || x_{n+1}^1 - x_{n+1}^2 ||.
\end{equation*}
\begin{remark}
    The method used to make the comparison has to be of higher order than the chosen one. In \cite{wadia2021optimization} the Euler discretization is compared with the Heun's one (also known as midpoint rule), which is more accurate.
\end{remark} 
The size of the next time step is then computed according to the rule
\begin{equation}
    h_{n+1} = \bigg(\frac{r}{\delta_n} \bigg)^{\theta/2} h_n,
\end{equation}
where $r$ is the desired error between $x_{n+1}^1$ and $x_{n+1}^2$ and $\theta$ is a gain. Concerning the choice of the parameters:
\begin{itemize}
    \item $\theta\in[0,2]$, and $\theta=0$ takes us back to a constant step size. As $\theta$ increases, a larger change in $h$ is allowed, potentially leading to convergence in fewer steps. The controller is sensitive with respect to changes in $\theta$ and choosing it larger than 0.01 may lead to oscillatory behaviour.
    \item $r\le0.6$ and changing it does not affect the controller as strongly as $\theta$. 
\end{itemize}
Some constraints on the minimum/maximum size of $h$ can be imposed according to the stability boundaries of the chosen method.

\section{Numerical results: optimization on the sphere}
In the next session we present some numerical results. As an example-problem we have chosen an optimization problem on the sphere. We want to find the minimum of the function
\begin{equation*}
    f(q) = q^T A q, 
\end{equation*}
where $A\in\mathbb{R}^{d\times d}$ is a symmetric matrix and $q$ is a point on $S^{d-1} = \{x \in \mathbb{R}^d \, : \, ||x||^2 = 1\}$. We then see $f$ as the potential energy of a Hamiltonian system and build the Hamiltonian function adding a momentum-dependent kinetic energy $T$:
\begin{equation*}
    H(q,p) = f(q) + T(p) = q^T A q + \frac{1}{2} p^T M^{-1} p,
\end{equation*}
where $M$ is the mass matrix.

\subsection{Results of the proposed method}
Here we present some results regarding the proposed method only. First we discuss the choice of the parameters. Then we compare fixed and adaptive stepsizes and the symplectic method of order 1 and 2.

\subsubsection{Choice of the parameters}
Some of the parameters are fixed for all the simulations. In particular:
\begin{itemize}
    \item[-] $A$ is built randomly from its largest and smallest eigenvalues (these extreme eigenvalues turn out to play a role in the results),
    \item[-] $d=10$,
    \item[-] $q_0 = (0,0,0,0,0,1,0,0,0,0)$,
    \item[-] $p_0 = (1,1,1,1,1,0,1,1,1,1)$,
    \item[-] $M = I$,
    \item[-] tolerance of $10^{-6}$ (the solution is compared with the one provided by a library, e.g. np.linalg.eig($A$)).
\end{itemize}
The parameters $\gamma$ and $h$, i.e. the dissipation and the stepsize, affect the performance of the method and we study what happens when changing them. We leave the discussion about $h$ for later, when we test also the adaptive stepsize method, and we first focus on $\gamma$. We define 
\begin{equation*}
    I_\lambda : = \lambda_{max} - \lambda_{min},
\end{equation*}
as we will refer often to this interval.

\textbf{Choice of the parameter of dissipation $\gamma$.}
We show some results of the method applied to a range of different matrices $A$ to get an intuition on how to choose $\gamma$, keeping $h$ fixed. By looking at the plots in Figure \ref{gamma plots} we observe that it is not easy to choose an optimal $\gamma$. As a general tendency we can say that when the eigenvalues range in a small interval $I_\lambda$, then also the values of $\gamma$ allowing for convergence are smaller than when $I_\lambda$ increases. Unfortunately, there seems to be no rule to choose it even when these two eigenvalues are fixed, as we can see in Figure \ref{gamma for same eigenvalues}. 
\begin{figure}[tb]
    \centering
    \includegraphics[width = 0.3\textwidth]{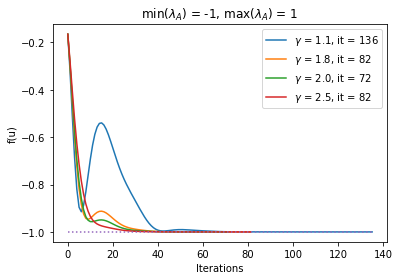}
    \includegraphics[width = 0.3\textwidth]{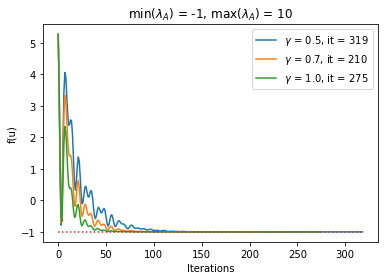}
    \includegraphics[width = 0.3\textwidth]{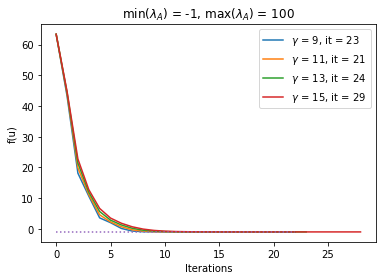}\\
    \includegraphics[width = 0.3\textwidth]{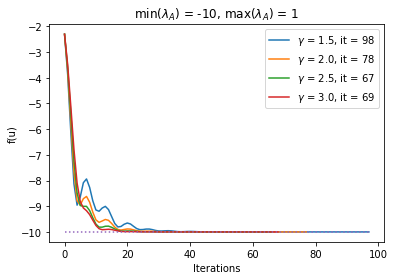}
    \includegraphics[width = 0.3\textwidth]{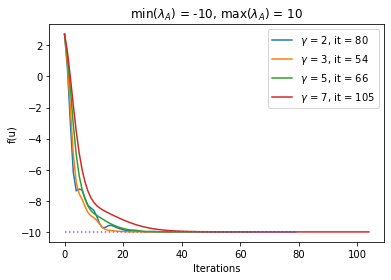}
    \includegraphics[width = 0.3\textwidth]{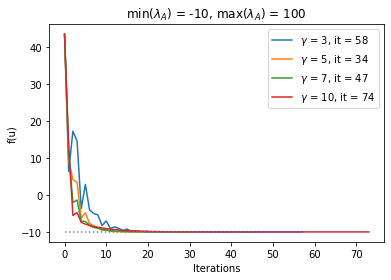}\\
    \includegraphics[width = 0.3\textwidth]{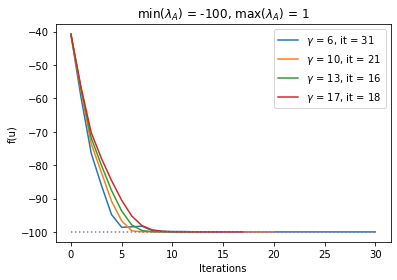}
    \includegraphics[width = 0.3\textwidth]{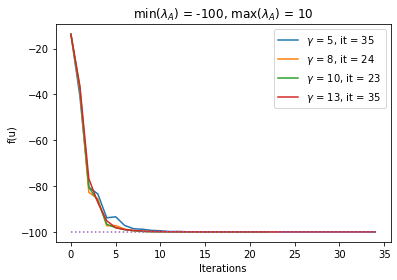}
    \includegraphics[width = 0.3\textwidth]{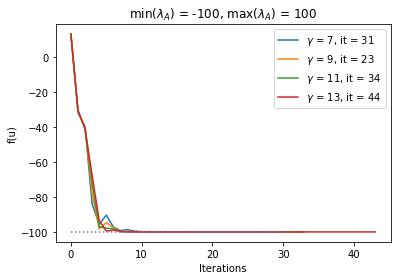}
    \caption{\small Results for the SM order 1, for different matrices $A$ and different parameters $\gamma$. The stepsize is 0.1 in all cases except for the last one where it is 0.09.}
    \label{gamma plots}
\end{figure}
\begin{figure}[tb]
    \centering
    \includegraphics[width = 0.3\textwidth]{plots/dissipation_hfix_SM1_Am1_AM1.png}
    \includegraphics[width = 0.3\textwidth]{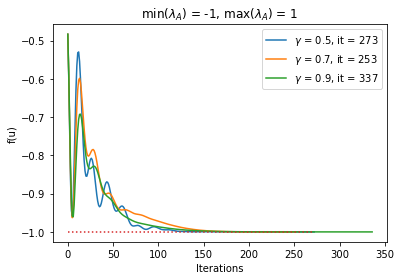}
    \includegraphics[width = 0.3\textwidth]{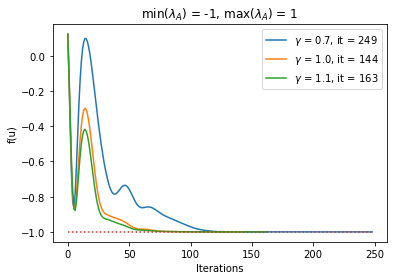}
    \caption{\small Results for the SM order 1, for different matrices $A$ with same $\lambda_{min}$ and $\lambda_{max}$ and different parameters $\gamma$. The stepsize is 0.1.}
    \label{gamma for same eigenvalues}
\end{figure}


\subsubsection{Comparing fixed and adaptive stepsize}
Before discussing comparison between the splitting method and its adaptive stepsize version, we need to comment the choice of the parameters in the latter one. In particular, $r$ (the desired error) is fixed to 0.06, while $\theta$ (the gain) is initially set to 0.001 and then increased until the code produces error or oscillations appear. As $h_0$ we consider 0.1 as long as $I_\lambda < 200$, and reduce it to $0.09$ for larger intervals $I_\lambda$. We can observe the results in Figure \ref{h plots}: the adaptive scheme works better than the original splitting method of order 1, when the latter is performed using $h = h_0$. A great difference in the number of iterations can be noticed especially when $I_\lambda$ is small, while as it increases, the methods become comparable. Also, as a general tendency we can say that smaller $I_\lambda$ requires more iterations than larger $I_\lambda$. Still in Figure \ref{h plots} are shown the results for the fixed stepsize method when increasing $h$ as much as possible for each particular case. The behavior of $h$ with respect to $I_\lambda$ appears to be the opposite than $\gamma$'s one: smaller $I_\lambda$ allows for larger stepsizes, whereas bigger $I_\lambda$ requires smaller $h$. 
\begin{figure}[tb]
    \centering
    \includegraphics[width = 0.3\textwidth]{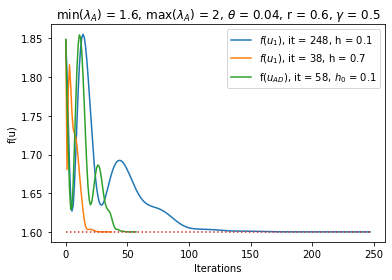}
    \includegraphics[width = 0.3\textwidth]{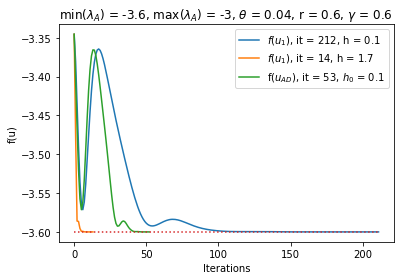}
    \includegraphics[width = 0.3\textwidth]{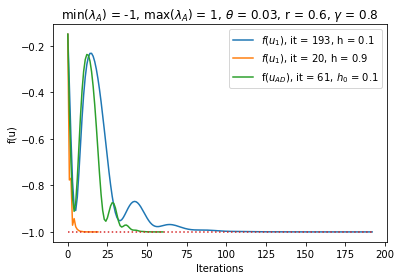}\\
    \includegraphics[width = 0.3\textwidth]{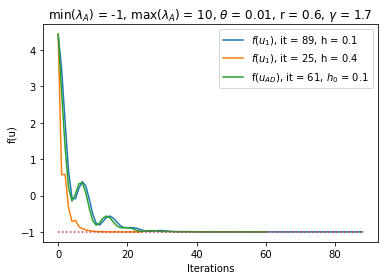}
    \includegraphics[width = 0.3\textwidth]{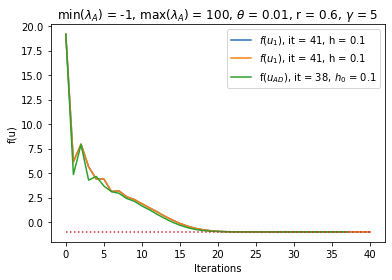}
    \includegraphics[width = 0.3\textwidth]{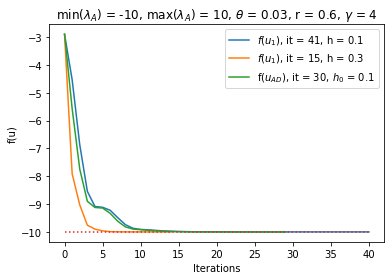}\\
    \includegraphics[width = 0.3\textwidth]{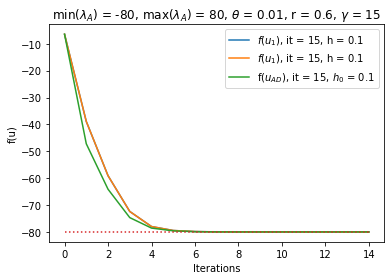}
    \includegraphics[width = 0.3\textwidth]{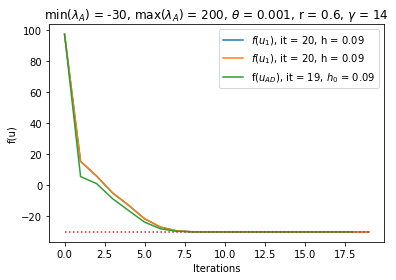}
    \includegraphics[width = 0.3\textwidth]{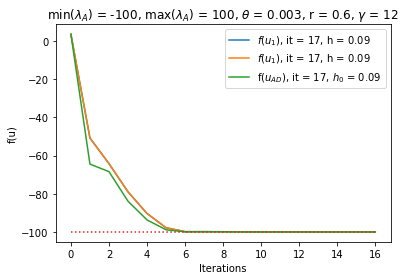}
    \caption{\small Results for the SM order 1 and for the adaptive scheme (AD), for different matrices $A$ and different stepsizes $h$.}
    \label{h plots}
\end{figure}

\subsubsection{Comparing SM order 1 and SM order 2}
As one can see from the results in Figure \ref{kinetic energy plots}, there is basically no difference in the number of iterations when using the splitting method of order 1 or 2. The evolution of $f$ is slightly different but as the aim is to optimize and not to compute an accurate trajectory, choosing either one works equally well.
\begin{figure}[tb]
    \centering
    \includegraphics[width = 0.3\textwidth]{plots/dissipation_hfix_SM1_Am1_AM1.png}
    \includegraphics[width = 0.3\textwidth]{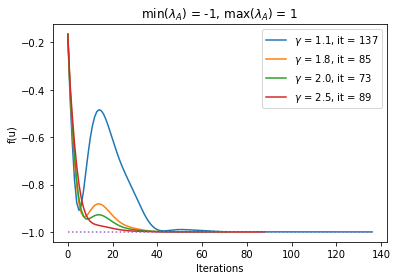}
    \caption{\small Results for the SM order 1 and of order 2.}
    \label{kinetic energy plots}
\end{figure}

\subsection{Comparing the proposed method with the usual gradient descent}
In this subsection we compare the proposed method and the usual gradient descent (GD) one. Before showing the results we present and discuss the GD method applied to our problem.

\subsubsection{GD method}
The GD method aims at finding the minimum of $f$ by following the direction given by its negative gradient. As we need to remain on the tangent plane of a sphere, a projection is performed. The resulting continuous equation reads
\begin{equation}
    \label{GD continuous}
    \Dot{q} = - (I - q q^T) \nabla f,
\end{equation}
and its discretization is
\begin{equation}
    \label{GD discretized}
    q_{n+1} = \frac{q_n - h \Big(I - q_n q_n^T \Big) \nabla f(q_n)}{||q_n - h \Big(I - q_n q_n^T \Big) \nabla f(q_n)||}.
\end{equation}
The normalization allows for the iterates to be on the unit sphere.

\subsubsection{Analysis of the GD method}
\label{Analysis of the GD method}
We try to do the analysis of the gradient descent scheme. To do so, we consider the function giving the next iterate
\begin{equation*}
    F(q) = \frac{q - h \big(I - q \, q^T \big) \nabla f(q)}{||q - h \big(I - q \, q^T \big) \nabla f(q)||}.
\end{equation*}
We aim at finding the minimizer $q^*$, that is a fixed point such that $F(q^*) = q^*$. We define the error at the $n^{th}$ iteration as
\begin{equation*}
    e_{n+1} = F(q_{n}) - F(q^*).
\end{equation*}
A Taylor expansion up to the first order (i.e. a linearization) of $F$ centered at the point $x_n$ gives
\begin{equation*}
    F(q^*) \approx F(q_n) + DF(q_n) \, (q^* - q_n),
\end{equation*}
leading to the approximation
\begin{equation*}
    e_n \approx DF(q_n) \, e_{n-1}.
\end{equation*}
It follows that the error decreases as the spectrum of the matrix $DF(q_n)$ is strictly smaller than 1. Letting $v=k q - 2hAq$, $c=q^TAq$, $k=1+2hq^TAq$, the Jacobian of $F$ is
\begin{equation*}
\begin{split}
    DF(q) &= \frac{1}{||v||} \bigg(k - 2h \,A + 4h \,q q^T A \bigg) \\
    & - \frac{1}{||v||^3} \bigg(k^3 qq^T - 2hk^2 Aqq^T - 8ckh^2\,qq^TA \bigg) \\
    & - \frac{1}{||v||^3} \bigg(4kh^2qq^TA^2 + 16ch^3Aqq^TA - 8h^3Aqq^TA^2\bigg).
\end{split}
\end{equation*}
The above expression is quite complex and, by studying it analytically, we could not get to any condition on the stepsize $h$ so that $\rho(DF(q))$ is always smaller than 1. We therefore opted for studying it numerically by looking at $\rho(DF(q_n))$ for different iterates. In Figure \ref{GD evolution of rho} we see the evolution of $\rho(DF(q))$ for different values of $h$. In all cases shown, $\rho(DF(q))>1$ for the first iterations, but decreases at each step. In the cases tested with $h<1$ (first two plots of Figure \ref{GD evolution of rho}), at some point $\rho(DF(q))$ becomes smaller than 0. Then
\begin{itemize}
    \item if $h$ is small enough, $\rho(DF(q))$ keeps decreasing until it reaches asymptotically a value between 0 and 1.
    \item There is a "critical" value of $h$, namely $h_c$: for $h>h_c$ we have that $\rho(DF(q))$ starts increasing from a certain iteration on and converges to a value between 0 and 1. 
    \item There is a "limiting" value of $h$, say $h_l$, for which $\rho(DF(q))$ tends to $1$, and if $h \ge h_l$ we do not have convergence.
\end{itemize}
\begin{figure}[tb]
    \centering
    \includegraphics[width = 0.3\textwidth]{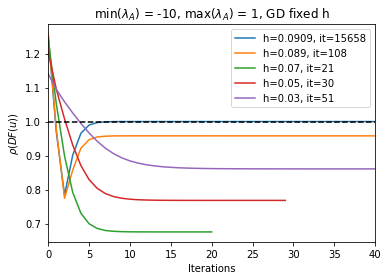}
    \includegraphics[width = 0.3\textwidth]{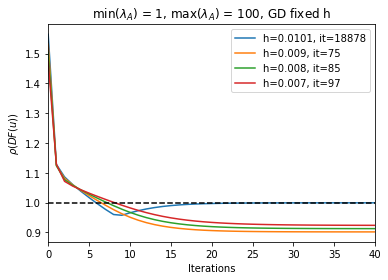}
    \includegraphics[width = 0.3\textwidth]{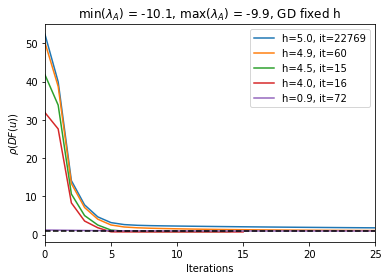}
    \caption{\small Examples of the evolution of $\rho(DF(q))$ for different values of $h$.}
    \label{GD evolution of rho}
\end{figure}
We left aside the cases in which we do not observe this behaviour (e.g. third plot in Figure \ref{GD evolution of rho}) as this study is just to find a rule for choosing $h$ and the result we get ends up holding true also for these cases.\\
Intuitively, one would like to chose $h$ the largest possible so that the number of iterations is reduced. Still, such $h$ must ensure convergence: we therefor need to find out the limiting value $h_l$. Experiments has led us to think that there is a relation between $h_l$ and the matrix $A$. In particular, we have found out that 
\begin{equation*}
    \rho(DF(q)) \cdot h \longrightarrow \frac{1}{\lambda_{max}(A) - \lambda_{min}(A)}
\end{equation*}
as $h \to h_l$, where $\lambda_{max}(A)$ and $\lambda_{min}(A)$ are the largest and smallest eigenvalues of $A$ respectively. Since we have seen above that $\rho(DF(q)) \to 1$, we get
\begin{equation}
    h_l =  \frac{1}{\lambda_{max}(A) - \lambda_{min}(A)}.
\end{equation}
In table \ref{limiting h table} we show some data we have collected. We have run multiple times the code for different matrices $A$, playing with the largest and smallest eigenvalues. For each $A$ we looked for the maximum value of $h$, $h_{max}$, allowing for convergence (according to our machines), that was a truncation of $h_l$. As, for such $h_{max}$, convergence was rather slow and the number of iteration was quite high, we tried to find a way to choose an \textit{optimal} $h$, $h_{opt}$, to improve our results. Truncating $h_l$ after its first nonzero digit works well (see Table \ref{limiting h table}). Still, there are some exceptions:
\begin{itemize}
    \item if $h_l \ge 1$ we take its integer part and subtract 0.1:
    \begin{equation*}
        \text{e.g. if } h_l = 3.4 \, \Rightarrow \text{ we take } h_{opt} = 2.9,
    \end{equation*}
    (3 would also work but it would not in the case $h_l = 3.0$, we made this choice to avoid having many cases).
    \item if $h_l < 1$ and the digit after its first nonzero digit is again zero, we proceed similarly as above:
    \begin{equation*}
        \text{e.g. if } h_l = 0.5 \, \Rightarrow \text{ we take } h_{opt} = 0.49.
    \end{equation*}
\end{itemize}
\begin{remark}
    Our choice of $h_{opt}$ is just a possibility: it is not proved to give the best $h$ possible, but it does ensure convergence in a reasonable amount of iterations. Also, we note that the number of iterations in Table \ref{limiting h table} is specific for each matrix $A$: choosing another one with the same $\lambda_{max}$ and $\lambda_{min}$ would lead to same values of $h$ but different number of iterations.
\end{remark}
\begin{table}[ht]
\small
\caption{\small Data resulting by running the GD method. The matrix $A$ has fixed largest and smallest eigenvalues and then is built randomly. The value $h_{max}$ is the largest values of $h$ allowing for convergence (according to our tests and our machines), whereas $h_{opt}$ is the value of $h$ which we propose to use. It ensures convergences and in most of the cases, it reduces significantly the number of iterations with respect to larger $h$.}
\label{limiting h table}
\renewcommand\arraystretch{1.2}
\noindent\[
\begin{array}{|c|c|c|c|c|c|c|}
\hline
\lambda_{min}(A)& \lambda_{max}(A) & h_l & h_{max} & \text{it} & h_{opt} & \text{it} \\
\hline
-100 & -10 & 0.0\overline{1} & 0.0111 & 1098 & 0.01 & 30 \\
-100 & 100 & 0.005 & 0.0049999 & 5388 & 0.0049 & 85 \\
-10 & 1 & 0.090909 & 0.0909 & 10394 & 0.09 & 110 \\
-10 & 100 & 0.0090909 & 0.00909 & 21472 & 0.009 & 300 \\
-1 & 1 & 0.5 & 0.49999 & 28256 & 0.49 & 72 \\
-1 & 100 & 0.0099009 & 0.0099 & 18292 & 0.009 & 62 \\
1 & 10 & 0.\overline{1} & 0.1109 & 1171 & 0.1 & 76 \\
1 & 100 & 0.010101 & 0.0101 & 21040 & 0.01 & 288 \\
-3 & 8 & 0.090909 & 0.0909 & 15413 & 0.09 & 202 \\
-27 & 58 & 0.01176 & 0.01175 & 1748 & 0.01 & 141 \\
-10.1 & -9.9 & 5 & 5 & 23403 & 4.9 & 35 \\
\hline
\end{array}
\]
\end{table}

\subsubsection{Comparing GD and SM} We are now ready to compare the proposed method with the usual gradient descent. In Figure \ref{comparison plots} we can look at the results when applying the SM of order one, the adaptive scheme (AD) and the GD method. For the SM we choose $h=0.1$, which is also the initial $h_0$ for AD; whereas for GD the stepsize is computed according to the rule proposed in section \ref{Analysis of the GD method}. Comparing the performances in terms of number of iterations: in most cases, the adaptive scheme and GD are either comparable, or the first is faster than the latter. These images confirm that the proposed method is a valid alternative to the usual gradient descent one.
\begin{figure}[tb]
    \centering
    \includegraphics[width = 0.3\textwidth]{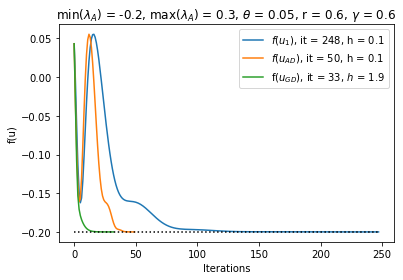}
    \includegraphics[width = 0.3\textwidth]{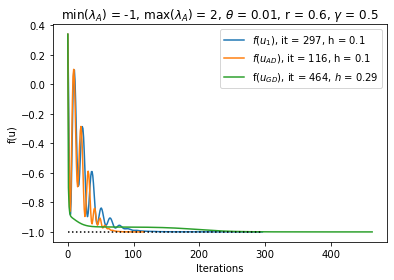}
    \includegraphics[width = 0.3\textwidth]{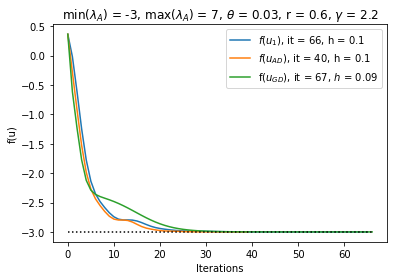}\\
    \includegraphics[width = 0.3\textwidth]{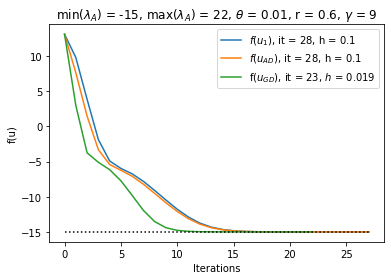}
    \includegraphics[width = 0.3\textwidth]{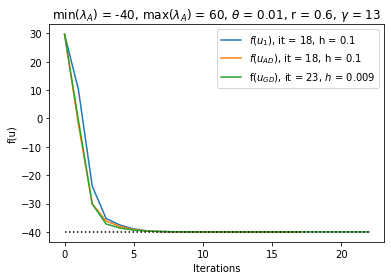}
    \includegraphics[width = 0.3\textwidth]{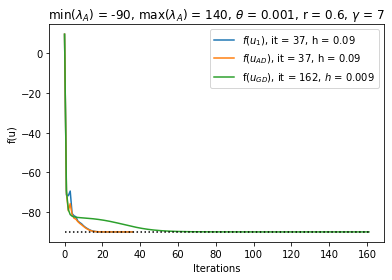}
    \caption{\small Results for the SM of order 1, the adaptive scheme (AD) and the GD method for different matrices $A$.}
    \label{comparison plots}
\end{figure}

\section{Conclusion and possible future directions}
In this work we have presented a conformal symplectic method suitable for the manifold setting. We have proposed two versions of it, of order one and two, and proved the respective properties. As a third option, we have formulated an adaptive stepsize version. We have tested all the schemes on an example, a function defined on a sphere, and in such setting they all resulted being a valid alternative to the usual gradient descent method. We tried numerically to optimize the parameters but, unfortunately, we could not establish a rule to choose them in this specific setting. It would be desirable to test the methods also on other well known examples, involving other compact manifolds such as the orthogonal group or the Stiefel manifold, and more oriented to the field of machine learning and deep learning. \\
A further work could be to build the conformal Hamiltonian system using a different kinetic energy, e.g. the relativistic one, and formulate a conformal symplectic method from it. That would mean to attempt the generalisation of the relativistic gradient descent method proposed in \cite{francca2020conformal} for the Euclidean case to the manifold setting.

\section{Acknowledgments}
I would like to thank my supervisors for the helpful discussions, advice and support.

\bibliographystyle{siamplain}
\bibliography{biblio}

\end{document}